\documentclass[11pt,reqno,a4paper]{amsart}
\usepackage{euscript}

\usepackage{hyperref}
\RequirePackage[dvipsnames]{xcolor} 
\definecolor{halfgray}{gray}{0.55} 
\definecolor{webgreen}{rgb}{0,0.5,0}
\definecolor{webbrown}{rgb}{.6,0,0} \hypersetup{%
  colorlinks=true, linktocpage=true, pdfstartpage=3,
  pdfstartview=FitV,%
  breaklinks=true, pdfpagemode=UseNone, pageanchor=true,
  pdfpagemode=UseOutlines,%
  plainpages=false, bookmarksnumbered, bookmarksopen=true,
  bookmarksopenlevel=1,%
  hypertexnames=true,
  pdfhighlight=/O,
  urlcolor=webbrown, linkcolor=RoyalBlue,
  citecolor=webgreen, 
  pdftitle={},%
  pdfauthor={},%
  pdfsubject={2000 MAthematical Subject Classification: Primary:},%
  pdfkeywords={},%
  pdfcreator={pdfLaTeX},%
  pdfproducer={LaTeX with hyperref}%
}

\newtheorem{definition}{Definition}
\newtheorem{theorem}{Theorem}
\newtheorem{proposition}{Proposition}
\newtheorem{corollary}{Corollary}
\newtheorem{lemma}{Lemma}
\newtheorem{remark}{Remark}

\renewcommand{\epsilon}{\varepsilon}
\renewcommand{\phi}{\varphi}

\DeclareMathOperator{\Ker}{Ker}
\def\Z{\mathbb{Z}}
\def\R{\mathbb{R}}
\def\cA{\EuScript{A}}

\def\Id{\text{\rm Id}}

\begin{document}

\title[Hyers-Ulam stability for hyperbolic random dynamics]{Hyers-Ulam stability for hyperbolic random dynamics}

\begin{abstract}
We prove that small nonlinear perturbations of random linear dynamics admitting a tempered exponential dichotomy have a random version of the shadowing property. As a consequence, if the exponential dichotomy is uniform, we get that the random linear dynamics is Hyers-Ulam stable. Moreover, we apply our results to study the conservation of Lyapunov exponents of the random linear dynamics subjected to nonlinear perturbations.
\end{abstract}

\author{Lucas Backes}
\address{Departamento de Matem\'{a}tica, Universidade Federal do Rio Grande do Sul, Av. Bento Gonc¸alves 9500, CEP 91509-900, Porto Alegre, RS, Brazil}
\email{lucas.backes@ufrgs.br}

\author{Davor Dragi\v cevi\'c}
\address{Department of Mathematics, University of Rijeka, Croatia}
\email{ddragicevic@math.uniri.hr}

\keywords{ Hyers-Ulam stability, hyperbolicity, random dynamical systems}
\subjclass[2010]{Primary: 37C50, 34D09; Secondary: 34D10.}
\maketitle

\maketitle

\section{Introduction}

The foundations of the theory of chaotic dynamical systems dates back to the work of Poincar\'e \cite{Poi90} and is now a well developed area of research. An important feature of chaotic dynamical systems, already observed by Poincar\'e, is the sensitivity to initial conditions: any small change to the initial condition may led to a large discrepancy in the output. This fact makes somehow complicated or even impossible the task of predicting the real trajectory of the system based on approximations. On the other hand, many chaotic systems, like uniformly hyperbolic dynamical systems \cite{An70, Bow75}, exhibit an important property stating that, even though a small error in the initial condition may led eventually to a large effect, there exists a true orbit with a slightly different initial condition that stays near the approximate trajectory. This property is known as the \emph{shadowing property} and its study was initiated in the seminal papers of Anosov~\cite{An70} and Bowen~\cite{Bow75}. Their original approach which was  based on the invariant manifold theory was later greatly simplified by Mayer and Sell~\cite{MS87} as well as Palmer~\cite{Pal88},
who presented purely analytic proofs of the shadowing property in the context of uniformly hyperbolic dynamics. We refer to the books~\cite{Pal00, Pil99} for more details and further references on the shadowing theory. 

More recently, numeruous authors have begun to study a similar problem in the context of differential or difference equations. More precisely, they are concerned  with formulating sufficient conditions under which one is able  to find an exact solution of a differential or difference equation in  a vicinity of an approximate solution. If the equation
possesses this property, we say that it exhibits \emph{Hyers-Ulam stability}. This terminology is used due  to the fact that Ulam~\cite{Ulam}   proposed a similar type of problem for functional equations,  whose partial solution was provided by Hyers~\cite{Hyers}.

As already mentioned, in the recent years many results dealing with Hyers-Ulam stability (discussing both its presence and lack of it) of differential and difference equations have been obtained. We in particular refer to the works of Brzd\c{e}k, Popa and Xu~\cite{BPX1, BPX2, BPX3, BPX4},  Jung~\cite{Jung},  Popa~\cite{Popa, Popa2}, Popa and Ra\c{s}a~\cite{PopaRasa,PopaRasa2}, 
Wang, Fe\v{c}kan and Tian~\cite{WFT}, Wang, Fe\v{c}kan and Zhou~\cite{WFZ}, 
Xu~\cite{X} 
as well as Xu,  Brzd\c{e}k and Zhang~\cite{XBZ}.

The relationship between hyperbolicity and Hyers-Ulam stability has been systematically studied in a series of papers by C. Bu\c{s}e and collaborators~\cite{BBT2,BRST,BLR}. Let us briefly describe the main results from~\cite{BBT2}. Assume that $A$ is a complex matrix of order $m$ and consider the associated dynamics
\begin{equation}\label{nad}
x_{n+1}=Ax_n, \quad n\ge 0.
\end{equation}
Then, it was established in~\cite{BBT2} that the following statements are equivalent:
\begin{itemize}
\item $A$ is hyperbolic, i.e. the spectrum of $A$ doesn't intersect the unit circle;
\item there exists $L>0$ such that for any $\epsilon >0$ and any sequence $(y_n)_{n\ge 0}\subset \mathbb C^m$ such that 
\[
\sup_{n\ge 0}\lVert y_{n+1}-Ay_n\rVert \le \epsilon, 
\]
there exists a  solution $(x_n)_{n\ge 0}\subset \mathbb C^m$ of~\eqref{nad} such that
\[
\sup_{n\ge 0}\lVert x_n-y_n\rVert \le L\epsilon.
\]
In other words, each approximate solution of~\eqref{nad} is close to an exact solution. 
\end{itemize}
We note that similar result concerned with a difference equation
\begin{equation}\label{nad2}
x_{n+1}=A_nx_n, \quad n\ge 0,
\end{equation}
where $(A_n)_{n\ge 0}$ is a periodic sequence of complex matrices were obtained in~\cite{BRST}. 

A recent advancement in this line of the research  was made in~\cite{BDp} (with the arguments based on~\cite{BD19}, which in turn are inspired by versions of shadowing lemma for nonautonomous dynamics~\cite{CLP89}). More precisely, we study~\eqref{nad2} in the infinite-dimensional case and without an assumption on the periodicity of the sequence $(A_n)_{n\ge 0}$. We have proved that if the sequence
$(A_n)_{n\ge 0}$ admits an exponential dichotomy then~\eqref{nad2} exhibits Hyers-Ulam stability (showing that also that the converse holds  in the finite-dimensional case). In fact, we show that the same conclusion holds true if we slightly  perturb linear dynamics~\eqref{nad2}, i.e. if we consider nonlinear dynamics
\[
x_{n+1}=A_nx_n+f_n(x_n), \quad n\ge 0,
\]
where $(f_n)_{n\ge 0}$ is a sequence of suitable Lipschitz (nonlinear) maps on $X$.

The objective of this paper is to establish results similar to those in~\cite{BDp} in the context of random dynamical systems. The theory of random dynamical systems is by now very well developed and its relevance in studying various real-life phenomena (such as for example those modelled by stochastic differential equations) is widely recognized. We refer to~\cite{A} for  a detailed exposition of this theory. 

As usual in random dynamics, we start with a base space which is a probability space $(\Omega, \mathcal F, \mathbb P)$ together with an $\mathbb P$-preserving  invertible  transformation $\sigma \colon \Omega \to \Omega$. Furthermore, we  have a family of linear operators $(A(\omega))_{\omega \in \Omega}$ acting on some Banach space $X$ and a family $(f_\omega)_{\omega \in \Omega}$ of (nonlinear) maps on $X$. For
  $\omega \in \Omega$, 
we consider the nonlinear dynamics
\begin{equation}\label{83:6}
x_{n+1}=A(\sigma^n (\omega))x_n+f_{\sigma^n (\omega)}(x_n), \quad n\in \Z.
\end{equation}
We show that if the cocycle generated by the linear part of~\eqref{83:6} admits a so-called tempered exponential dichotomy and under suitable assumptions for nonlinear maps $f_\omega$, in a vicinity of each suitable approximate solution of~\eqref{83:6} we can find an exact solution. In the particular case, when the linear part of~\eqref{83:6} admits a uniform exponential dichotomy, our results imply that~\eqref{83:6} is Hyers-Ulam stable. 

In contrast to~\cite{BDp}, besides considering a random dynamics (and not nonautonomous dynamics given  by a sequence of maps), we also:
\begin{itemize}
\item deal with the situation when the linear part of~\eqref{83:6} admits a tempered-exponential dichotomy, i.e.  it is nonuniformly hyperbolic. The concept of nonuniform hyperbolicity originated in the landmark works of  Oseledets~\cite{Osel} and particularly 
 Pesin~\cite{P1} and proved to be a  nontrivial and far-reaching extension of the classical theory of uniformly hyperbolic dynamical systems
initiated by Smale~\cite{Smale}.  For extension of this theory to the case of infinite-dimensional dynamics we refer to the works of Ruelle~\cite{Ruelle}, Ma\~{n}\'{e}~\cite{Mane}, Thieullen~\cite{T}, Lian and Lu~\cite{LL}, Zhou, Lu and Zhang~\cite{ZLZ} and additional references therein.
\item consider a broader class of approximate solutions of~\eqref{83:6} than those in~\cite{BDp}. More precisely, we now deal with situations when the error in the one step iteration of the dynamics is no longer uniform over time. 
\end{itemize}
We stress that those novelties require us to substantially modify our previous arguments from~\cite{BDp}.

\section{Preliminaries}
Let $X$ be an arbitrary Banach space and let $B(X)$ denote the space of all bounded operators on $X$. Furthermore, consider a probability space $(\Omega, \mathcal F, \mathbb P)$ together with a $\mathbb P$-preserving  invertible  transformation $\sigma \colon \Omega \to \Omega$. We will assume that $\mathbb P$ is ergodic. 

Let $A\colon \Omega \to B(X)$ be a strongly measurable map, i.e. $\omega \to A(\omega)x$ is a measurable map for each $x\in X$. We consider the associated cocycle $\cA \colon \Omega \times \mathbb N_0 \to B(X)$ defined by
\[
\cA(\omega, n)=\begin{cases}
\Id & \text{if $n=0$;}\\
A(\sigma^{n-1}(\omega)) \cdots A(\sigma (\omega))A(\omega)  & \text{if $n\in \mathbb N$.}
\end{cases}
\]
Observe that
\[
\cA(\omega, n+m)=\cA(\sigma^m (\omega), n) \cA(\omega, m), \quad \text{for $\omega \in \Omega$ and $m, n\in \mathbb N_0$.}
\]

We recall some notions of central importance to our results. 
\begin{definition}
 A measurable map $K\colon \Omega \to (0, \infty)$ is said to be a \emph{tempered random variable} if
\[
\lim_{n\to \pm \infty} \frac 1 n \log K(\sigma^n(\omega))=0, \quad \text{for $\mathbb P$-a.e. $\omega \in \Omega$.}
\]
\end{definition}
\begin{definition}\label{xxv}
We say that $\cA$ admits a \emph{tempered exponential dichotomy} if there exist $\lambda >0$, a tempered random variable $K\colon \Omega \to (0, \infty)$, a $\sigma$-invariant set $\Omega' \subset \Omega$ with $\mathbb P(\Omega')=1$ and a strongly measurable map $\Pi \colon \Omega \to B(X)$ such that for $\omega \in \Omega'$:
\begin{enumerate}
\item $\Pi(\omega)$ is a projection on $X$;
\item \begin{equation}\label{proj}\Pi(\sigma^n (\omega))\cA(\omega, n)=\cA(\omega, n)\Pi(\omega) \quad \text{for $n\in \mathbb N$;}
\end{equation}
\item for $n\in \mathbb N$,
\[
\cA(\omega, n)\rvert_{\Ker \Pi(\omega)} \colon \Ker \Pi(\omega) \to \Ker \Pi(\sigma^n (\omega))
\]
is invertible;
\item 
\begin{equation}\label{td1}
\lVert \cA(\omega, n)\Pi(\omega)\rVert \le K(\omega)e^{-\lambda n}, \quad n\ge 0
\end{equation}
and
\begin{equation}\label{td2}
\lVert \cA(\omega, -n)(\Id-\Pi(\omega))\rVert \le K(\omega)e^{-\lambda n}, \quad n\ge 0,
\end{equation}
where 
\[
\cA(\omega, -n):=\bigg{(}\cA(\sigma^{-n}(\omega), n)\rvert_{\Ker \Pi (\sigma^{-n} (\omega))} \bigg{)}^{-1}.
\]
\end{enumerate}
\end{definition}

\begin{remark}
It was proved in~\cite{BD19-2} that if $\cA$ satisfies the assumptions of the version of the  Multiplicative ergodic theorem established in~\cite{GTQ} and has all nonzero Lyapunov exponents that then it admits a tempered exponential dichotomy. Hence, the notion of a tempered exponential dichotomy is ubiquitous from the ergodic theory point of view. 

\end{remark}

From now on, we assume that $\cA$ is a cocycle that admits a tempered exponential dichotomy. Let $f_\omega\colon X \to X$, $\omega \in \Omega$, be a family of (nonlinear) maps. 
 For $\omega \in \Omega$, we consider the associated nonlinear and nonautonomous  dynamics~\eqref{83:6}. 
Observe that~\eqref{83:6} can be written as
\begin{equation}\label{nde}
x_{n+1}=F_{\sigma^n (\omega)}(x_n)\quad n\in \Z,
\end{equation}
where
\[
F_\omega:=A(\omega)+f_\omega. 
\]

We now introduce a family of adapted norms. For $\omega \in \Omega'$ and $x\in X$, let
\[
\lVert x\rVert_{\omega}:=\sup_{n\ge 0}(\lVert \cA(\omega, n)\Pi(\omega)x\rVert e^{\lambda n})+\sup_{n\ge 0}(\lVert \cA(\omega, -n)(\Id-\Pi(\omega))x\rVert e^{\lambda n}).
\]
Note that it follows from~\eqref{td1} and~\eqref{td2}  that
\begin{equation}\label{ln}
\lVert x\rVert \le \lVert x\rVert_\omega \le 2K(\omega)\lVert x\rVert, \quad \text{for $\omega \in \Omega'$ and $x\in X$.}
\end{equation}

We need the following classical lemma whose proof we include for the sake of completness. 
\begin{lemma}
We have that for each $\omega \in \Omega'$, $n\ge 0$ and $x\in X$,
\begin{equation}\label{ln1}
\lVert \cA(\omega, n)\Pi(\omega)x\rVert_{\sigma^n (\omega)}\le e^{-\lambda n}\lVert x\rVert_\omega
\end{equation}
and
\begin{equation}\label{ln2}
\lVert \cA(\omega, -n)(\Id-\Pi(\omega))x\rVert_{\sigma^{-n} (\omega)}\le e^{-\lambda n}\lVert x\rVert_\omega.
\end{equation}
\end{lemma}

\begin{proof}
We have that
\[
\begin{split}
\lVert \cA(\omega, n)\Pi(\omega)x\rVert_{\sigma^n (\omega)} &=\sup_{m\ge 0}(\lVert \cA(\sigma^{n}(\omega), m)\Pi(\sigma^n (\omega))\cA(\omega, n)\Pi(\omega)x\rVert e^{\lambda m}) \\
&=\sup_{m\ge 0} (\lVert \cA(\omega, n+m) \Pi(\omega)x\rVert e^{\lambda m}) \\
&=e^{-\lambda n}\sup_{m\ge 0} (\lVert \cA(\omega, n+m) \Pi(\omega)x\rVert e^{\lambda (m+n)}) \\
&\le e^{-\lambda n} \lVert x\rVert_\omega,
\end{split}
\]
and hence~\eqref{ln1} holds. Similarly, we have 
\[
\begin{split}
&\lVert \cA(\omega, -n)(\Id-\Pi(\omega))x\rVert_{\sigma^{-n} (\omega)} \\
&=\sup_{m\ge 0} (\lVert \cA(\sigma^{-n} (\omega), -m)(\Id-\Pi(\sigma^{-n}(\omega)))\cA(\omega, -n)(\Id-\Pi(\omega))x\rVert e^{\lambda m})\\
&=\sup_{m\ge 0}(\lVert \cA(\omega, -m-n)(\Id-\Pi(\omega))x\rVert e^{\lambda m})\\
&=e^{-\lambda n}\sup_{m\ge 0}(\lVert \cA(\omega, -m-n)(\Id-\Pi(\omega))x\rVert e^{\lambda (m+n)})\\
&\le e^{-\lambda n}\lVert x\rVert_\omega,
\end{split}
\]
and consequently~\eqref{ln2} also holds. 
\end{proof}
Before we state our first result, we will introduce  additional terminology.
\begin{definition}
For $r>0$ and a two-sided sequence of positive real numbers $\delta \colon \Z \to (0, \infty)$, we say that $\delta$ is \emph{$r$-admissible} if:
\[
\sup_{n\in \Z}\max \bigg{\{}\frac{\delta (n+1)}{\delta (n)}, \frac{\delta (n)}{\delta (n+1)} \bigg{\}}\le r.
\]
\end{definition}

\section{Main results}
The following is our first result. 
\begin{theorem}\label{t1}
Assume that $\cA$ admits a tempered exponential dichotomy and let $\Omega'\subset \Omega$, $\lambda >0$ and $K\colon \Omega \to (0, \infty)$ be as in the Definition~\ref{xxv}.  Furthermore, suppose that $\epsilon>0$, $c\ge 0$ are such that $\epsilon \le \lambda$ and that
\begin{equation}\label{contraction}
2ce^{\lambda-\epsilon} \frac{1+e^{-\epsilon}}{1-e^{-\epsilon}}<1. 
\end{equation}
Finally, we assume that
\begin{equation}\label{non}
\lVert f_\omega(x)-f_\omega(y)\rVert \le \frac{c}{K(\sigma (\omega))}\lVert x-y\rVert, \quad \text{for $\omega \in \Omega'$ and $x, y\in X$.}
\end{equation}
Then, there exists $L=L(\lambda, \epsilon, c)>0$ such that for every  $e^{\lambda-\epsilon}$-admissible sequence 
$\delta \colon \Z \to (0, \infty)$,  
 every  $\omega \in \Omega'$ and an arbitrary  sequence  $(y_n)_{n\in \Z}\subset X$ satisfying
\begin{equation}\label{pseudo}
\lVert y_n-F_{\sigma^{n-1}(\omega)}(y_{n-1})\rVert \le \frac{\delta (n)}{2K(\sigma^n (\omega))} \quad \text{for $n\in \Z$,}
\end{equation}
 there is a solution $(x_n)_{n\in \Z}$ of~\eqref{nde} with the property that
\begin{equation}\label{shad}
\lVert x_n-y_n\rVert \le L\delta (n), \quad \text{for each $n\in \Z$.}
\end{equation}
\end{theorem}

\begin{proof}
We will split the proof into several lemmas. 
Let us begin by introducing some auxiliarly notation. Set
\[
Y_{\delta, \infty}:=\bigg{\{} \mathbf z=(z_n)_{n\in \Z}\subset X: \sup_{n\in \Z}\big{(}\delta (n)^{-1}\lVert z_n\rVert_{\sigma^n(\omega)} \big{)}<\infty \bigg{\}}.
\]
It is easy to verify that  $Y_{\delta, \infty}$ is a Banach space with respect to the norm
\[
\lVert \mathbf z\rVert_{\delta, \infty}:=\sup_{n\in \Z}\big{(}\delta (n)^{-1}\lVert z_n\rVert_{\sigma^n(\omega)} \big{)}.
\]
Furthermore, we define $\Gamma_\omega \colon Y_{\delta, \infty} \to Y_{\delta, \infty}$ by 
\[
\begin{split}
(\Gamma_\omega \mathbf z)_n &=\sum_{k=0}^\infty \cA(\sigma^{n-k}(\omega),k)\Pi(\sigma^{n-k}(\omega))z_{n-k} \\
&\phantom{=}-\sum_{k=1}^\infty \cA(\sigma^{n+k}(\omega), -k)(\Id-\Pi(\sigma^{n+k}(\omega)))z_{n+k}.
\end{split}
\]
We need the following auxiliarly result. 
\begin{lemma}\label{l2}
We have that $\Gamma_\omega$ is a well-defined and bounded linear operator on $Y_{\delta, \infty}$. Furthermore, 
\begin{equation}\label{bms}
\lVert \Gamma_\omega \rVert \le \frac{1+e^{-\epsilon}}{1-e^{-\epsilon}}.
\end{equation}
\end{lemma}

\begin{proof}[Proof of the lemma]
Obviously $\Gamma_\omega$ is linear. Moreover, 
observe that for each $\mathbf z=(z_n)_{n\in \Z} \in Y_{\delta, \infty}$, it follows from~\eqref{ln1} and  \eqref{ln2}  that 
\[
\begin{split}
\delta(n)^{-1}\lVert (\Gamma_\omega \mathbf z)_n\rVert_{\sigma^n (\omega)} &\le \delta(n)^{-1}\sum_{k=0}^\infty e^{-\lambda k}\lVert z_{n-k}\rVert_{\sigma^{n-k}(\omega)} \\
&\phantom{\le}+\delta(n)^{-1}\sum_{k=1}^\infty e^{-\lambda k}\lVert z_{n+k}\rVert_{\sigma^{n+k}(\omega)} \\
&=\sum_{k=0}^\infty e^{-\lambda k}\frac{\delta (n-k)}{\delta (n)}\delta(n-k)^{-1}\lVert z_{n-k}\rVert_{\sigma^{n-k}(\omega)} \\
&\phantom{=}+\sum_{k=1}^\infty e^{-\lambda k}\frac{\delta (n+k)}{\delta (n)}\delta (n+k)^{-1}\lVert z_{n+k}\rVert_{\sigma^{n+k}(\omega)}\\
&\le \sum_{k=0}^\infty e^{-\lambda k}e^{(\lambda-\epsilon)k}\delta(n-k)^{-1}\lVert z_{n-k}\rVert_{\sigma^{n-k}(\omega)} \\
&\phantom{\le}+\sum_{k=1}^\infty e^{-\lambda k} e^{(\lambda-\epsilon)k}\delta (n+k)^{-1}\lVert z_{n+k}\rVert_{\sigma^{n+k}(\omega)}\\
&\le \bigg{(} \sum_{k=0}^\infty e^{-\epsilon k}+\sum_{k=1}^\infty e^{-\epsilon k}\bigg{)}\lVert \mathbf z\rVert_{\delta, \infty}\\
&=\frac{1+e^{-\epsilon}}{1-e^{-\epsilon}}\lVert \mathbf z\rVert_{\delta, \infty}.
\end{split}
\]
Hence, by taking supremum over all $n\in \Z$, we obtain that
\[
\lVert \Gamma_\omega \mathbf z\rVert_{\delta, \infty} \le \frac{1+e^{-\epsilon}}{1-e^{-\epsilon}}\lVert \mathbf z\rVert_{\delta, \infty}.
\]
We conclude that $\Gamma_\omega$ is well-defined and bounded operator.  In addition, \eqref{bms} holds. 
\end{proof}

In the following lemma we explain the role of the operator $\Gamma_\omega$. 
\begin{lemma}\label{l3}
For $\omega\in \Omega'$ and $\mathbf z=(z_n)_{n\in \Z}, \mathbf w=(w_n)_{n\in \Z}\in Y_{\delta, \infty}$, the following statements are equivalent:
\begin{enumerate}
\item $\Gamma_\omega \mathbf z=\mathbf w$;
\item for $n\in \Z$, 
\begin{equation}\label{10:34}
w_n-A(\sigma^{n-1}(\omega))w_{n-1}=z_n.
\end{equation}
\end{enumerate}
\end{lemma}

\begin{proof}[Proof of the lemma]
Let us assume that $\Gamma_\omega \mathbf z=\mathbf w$. For each $n\in \Z$, we have that
\[
\begin{split}
&w_n-A(\sigma^{n-1}(\omega))w_{n-1} \\
&=\sum_{k=0}^\infty \cA(\sigma^{n-k}(\omega),k)\Pi(\sigma^{n-k}(\omega))z_{n-k}\\
&\phantom{=}-A(\sigma^{n-1}(\omega))\sum_{k=0}^\infty \cA(\sigma^{n-k-1}(\omega),k)\Pi(\sigma^{n-k-1}(\omega))z_{n-k-1}\\
&\phantom{=}-\sum_{k=1}^\infty \cA(\sigma^{n+k}(\omega), -k)(\Id-\Pi(\sigma^{n+k}(\omega)))z_{n+k}\\
&\phantom{=}+A(\sigma^{n-1}(\omega))\sum_{k=1}^\infty \cA(\sigma^{n+k-1}(\omega), -k)(\Id-\Pi(\sigma^{n+k-1}(\omega)))z_{n+k-1}\\
&=\sum_{k=0}^\infty \cA(\sigma^{n-k}(\omega),k)\Pi(\sigma^{n-k}(\omega))z_{n-k}\\
&\phantom{=}-\sum_{k=0}^\infty \cA(\sigma^{n-k-1}(\omega),k+1)\Pi(\sigma^{n-k-1}(\omega))z_{n-k-1}\\
&\phantom{=}-\sum_{k=1}^\infty \cA(\sigma^{n+k}(\omega), -k)(\Id-\Pi(\sigma^{n+k}(\omega)))z_{n+k}\\
&\phantom{=}+\sum_{k=1}^\infty \cA(\sigma^{n+k-1}(\omega), -(k-1))(\Id-\Pi(\sigma^{n+k-1}(\omega)))z_{n+k-1}\\
&=\Pi(\sigma^n (\omega))z_n+(\Id-\Pi(\sigma^n (\omega)))z_n \\
&=z_n.
\end{split}
\]
Thus, 
\[
w_n-A(\sigma^{n-1}(\omega))w_{n-1}=z_n, \quad \text{for $n\in \Z$.}
\]

Let us now establish the converse.  Assume that~\eqref{10:34} holds for each $n\in \Z$. Set
\[
w_n^s:=\Pi(\sigma^n(\omega))w_n \quad \text{and} \quad w_n^u:=w_n-w_n^s.
\]
It follows from~\eqref{proj} and~\eqref{10:34} that  
\[
\begin{split}
w_n^s &= \Pi(\sigma^n(\omega))A(\sigma^{n-1} (\omega))w_{n-1}+\Pi(\sigma^n(\omega))z_n \\
&=A(\sigma^{n-1} (\omega))\Pi(\sigma^{n-1} (\omega))w_{n-1}+\Pi(\sigma^n(\omega))z_n \\
&=A(\sigma^{n-1} (\omega))\Pi(\sigma^{n-1} (\omega)) A(\sigma^{n-2}(\omega))w_{n-2} \\
&\phantom{=}+A(\sigma^{n-1} (\omega))\Pi(\sigma^{n-1} (\omega))z_{n-1}+\Pi(\sigma^n(\omega))z_n \\
&=\cA(\sigma^{n-2}(\omega), 2)\Pi(\sigma^{n-2}(\omega))w_{n-2} \\
&\phantom{=}+\cA(\sigma^{n-1} (\omega), 1)\Pi(\sigma^{n-1} (\omega))z_{n-1}+\cA(\sigma^n (\omega), 0)\Pi(\sigma^n(\omega))z_n. \\
\end{split}
\]
Proceeding inductively, we find that that 
\[
\begin{split}
w_n^s &=\cA(\sigma^{n-k}(\omega), k)\Pi(\sigma^{n-k}(\omega))w_{n-k} \\
&\phantom{=}+\sum_{j=0}^{k-1}\cA(\sigma^{n-j} (\omega), j)\Pi(\sigma^{n-j} (\omega))z_{n-j},
\end{split}
\]
for each $k\in \mathbb N$. Passing to the limit when $k\to \infty$ and using~\eqref{ln1}, we conclude that
\begin{equation}\label{11:54}
w_n^s=\sum_{k=0}^\infty \cA(\sigma^{n-k}(\omega),k)\Pi(\sigma^{n-k}(\omega))z_{n-k}, \quad \text{for $n\in \Z$.}
\end{equation}
Similarly, one can show that
\begin{equation}\label{11:55}
w_n^u=-\sum_{k=1}^\infty \cA(\sigma^{n+k}(\omega), -k)(\Id-\Pi(\sigma^{n+k}(\omega)))z_{n+k}, \quad \text{for $n\in \Z$.}
\end{equation}
By~\eqref{11:54} and~\eqref{11:55}, we have that $\Gamma_\omega \mathbf z=\mathbf w$ and the proof of the lemma is complete. 
\end{proof}
For $n\in \Z$, we define $g_n\colon X\to X$ by 
\[
\begin{split}
g_n(x) &=f_{\sigma^n (\omega)}(x+y_n)-f_{\sigma^n (\omega)}(y_n)+F_{\sigma^n (\omega)}(y_n)-y_{n+1} \\
&=f_{\sigma^n (\omega)}(x+y_n)+A(\sigma^n(\omega))y_n-y_{n+1},
\end{split}
\]
for $x\in X$. Furthermore, 
for $\mathbf z=(z_n)_{n\in \Z}\in Y_{\delta, \infty}$, set
\[
(S(\mathbf z))_n:=g_{n-1}(z_{n-1}), \quad n\in \Z.
\]
We need the following estimate. 
\begin{lemma}\label{l4}
For $\mathbf z^i=(z_n^i)_{n\in  \Z} \in Y_{\delta, \infty}$, $i=1, 2$ we have that
\[
\lVert S(\mathbf z^1)-S(\mathbf z^2)\rVert_{\delta, \infty} \le 2ce^{\lambda-\epsilon}\lVert \mathbf z^1-\mathbf z^2\rVert_{\delta, \infty}.
\]
\end{lemma}

\begin{proof}[Proof of the lemma]
By~\eqref{ln} and \eqref{non}, we have that 
\[
\begin{split}
&\lVert S(\mathbf z^1)-S(\mathbf z^2)\rVert_{\delta, \infty} \\
&=\sup_{n\in \Z} \bigg{(}\delta (n)^{-1}\lVert g_{n-1}(z_{n-1}^1)-g_{n-1}(z_{n-1}^2)\rVert_{\sigma^n (\omega)} \bigg{)}\\
&=\sup_{n\in \Z} \bigg{(}\delta (n)^{-1}\lVert f_{\sigma^{n-1}(\omega)}(z_{n-1}^1+y_{n-1})-f_{\sigma^{n-1}(\omega)}(z_{n-1}^2+y_{n-1})\rVert_{\sigma^n (\omega)} \bigg{)}\\
&\le \sup_{n\in \Z}\bigg{(}2K(\sigma^n (\omega)) \delta (n)^{-1}\lVert f_{\sigma^{n-1}(\omega)}(z_{n-1}^1+y_{n-1})-f_{\sigma^{n-1}(\omega)}(z_{n-1}^2+y_{n-1})\rVert \bigg{)}\\
&\le \sup_{n\in \Z}\bigg{(}2K(\sigma^n (\omega)) \delta (n)^{-1}\frac{c}{K(\sigma^n (\omega))} \lVert z_{n-1}^1- z_{n-1}^2\rVert \bigg{)}\\
&\le \sup_{n\in \Z}\bigg{(}2c\delta (n)^{-1}\lVert z_{n-1}^1- z_{n-1}^2\rVert_{\sigma^{n-1}(\omega)} \bigg{)}\\
&=\sup_{n\in \Z}\bigg{(}2c\frac{\delta(n-1)}{\delta (n)} \delta (n-1)^{-1}\lVert z_{n-1}^1- z_{n-1}^2\rVert_{\sigma^{n-1}(\omega)} \bigg{)}\\
&\le 2ce^{\lambda-\epsilon}\sup_{n\in \Z}\bigg{(} \delta (n-1)^{-1}\lVert z_{n-1}^1- z_{n-1}^2\rVert_{\sigma^{n-1}(\omega)} \bigg{)}\\
&=2ce^{\lambda-\epsilon}\lVert \mathbf z^1-\mathbf z^2\rVert_{\delta, \infty},
\end{split}
\]
and the conclusion of the lemma follows. 
\end{proof}
For $\mathbf z\in Y_{\delta, \infty}$, let
\begin{equation}\label{eq: op T}
T(\mathbf z)=\Gamma_\omega S(\mathbf z).
\end{equation}
It follows from Lemmas~\ref{l2} and~\ref{l4} that 
\begin{equation}\label{954}
\lVert T(\mathbf z^1)-T(\mathbf z^2)\rVert_{\delta, \infty} \le 2ce^{\lambda-\epsilon} \frac{1+e^{-\epsilon}}{1-e^{-\epsilon}} \lVert \mathbf z^1-\mathbf z^2\rVert_{\delta, \infty}, \quad  \text{for $\mathbf z^i \in Y_{\delta, \infty}$, $i=1,2$.}
\end{equation}
Let $L>0$ be such that it satisfies
\[
 \frac{1+e^{-\epsilon}}{1-e^{-\epsilon}}+2ce^{\lambda-\epsilon} \frac{1+e^{-\epsilon}}{1-e^{-\epsilon}}L=L.
\]
Set
\[
D:=\{ \mathbf z\in Y_{\delta, \infty}: \lVert \mathbf z\rVert_{\delta, \infty} \le L\}.
\]
The final ingredient of the proof is the following lemma.
\begin{lemma}\label{l5}
We have that  $T(D)\subset D$.
\end{lemma}

\begin{proof}[Proof of the lemma]
We begin by observing that~\eqref{ln} and~\eqref{pseudo} imply that 
\[
\begin{split}
\lVert S(\mathbf 0)\rVert_{\delta, \infty}&=\sup_{n\in \Z} (\delta(n)^{-1}\lVert y_n-F_{\sigma^{n-1}(\omega)}(y_{n-1})\rVert_{\sigma^n (\omega)})\\
&\le  \sup_{n\in \Z} (\delta(n)^{-1}2K(\sigma^n (\omega))\lVert y_n-F_{\sigma^{n-1}(\omega)}(y_{n-1})\rVert)\\
&\le 1. 
\end{split}
\]
Hence, it follows from~\eqref{954} that for any $\mathbf z\in D$, we have that 
\[
\begin{split}
\lVert T(\mathbf z)\rVert_{\delta, \infty} &\le \lVert T(\mathbf 0)\rVert_{\delta, \infty}+\lVert T(\mathbf z)-T(\mathbf 0)\rVert_{\delta, \infty} \\
&\le \lVert \Gamma_\omega \rVert \cdot \lVert S(\mathbf 0)\rVert_{\delta, \infty}+\lVert T(\mathbf z)-T(\mathbf 0)\rVert_{\delta, \infty} \\
&\le \frac{1+e^{-\epsilon}}{1-e^{-\epsilon}}+2ce^{\lambda-\epsilon} \frac{1+e^{-\epsilon}}{1-e^{-\epsilon}} \lVert \mathbf z\rVert_{\delta, \infty}\\
&\le \frac{1+e^{-\epsilon}}{1-e^{-\epsilon}}+2ce^{\lambda-\epsilon} \frac{1+e^{-\epsilon}}{1-e^{-\epsilon}}L\\
&=L,
\end{split}
\]
and thus the desired conclusion holds. 
\end{proof}
By~\eqref{954} and Lemma~\ref{l5}, we have that $T$ is  a contraction on $D$ and thus it has a unique fixed point $\mathbf z=(z_n)_{n\in \Z}\in D$. Hence, $\Gamma_\omega S(\mathbf z)=\mathbf z$ and consequently Lemma~\ref{l3} implies that 
\[
\begin{split}
&z_n-A(\sigma^{n-1}(\omega))z_{n-1} \\
&=f_{\sigma^{n-1} (\omega)}(z_{n-1}+y_{n-1})-f_{\sigma^{n-1} (\omega)}(y_{n-1})+F_{\sigma^{n-1} (\omega)}(y_{n-1})-y_{n} \\
&=f_{\sigma^{n-1} (\omega)}(z_{n-1}+y_{n-1})+A(\sigma^{n-1}(\omega))y_{n-1}-y_n,
\end{split}
\]
for each $n\in \Z$. Therefore, the sequence $(x_n)_{n\in \Z}$ defined by
\[
x_n=y_n+z_n \quad n\in \Z,
\]
is a solution of~\eqref{nde}. Moreover, we have (using~\eqref{ln}) that 
\[
\sup_{n\in \Z} (\delta (n)^{-1}\lVert x_n-y_n\rVert) \le \sup_{n\in \Z} (\delta (n)^{-1}\lVert x_n-y_n \rVert_{\sigma^n (\omega)})=\lVert \mathbf z\rVert_{\delta, \infty} \le L, 
\]
which readily implies~\eqref{shad}.
\end{proof}

Let us discuss the relationship between Theorem \ref{t1} and previous results available in the literature.
\begin{remark}\label{remark 2}
Along with results due to Katok, Hirayama and many others (see \cite{BP07,Pal00,Pil99}), Theorem \ref{t1} can be regarded as a version of the shadowing property for nonuniformly hyperbolic dynamics. Nevertheless, to the best of our knowledge, in all the previous results there are some stronger assumptions like differentiability, invertibility, boundedness and/or compactness or finite dimensionality, that are not present in Theorem \ref{t1}. Moreover, the error allowed in the one step iteration of the dynamics \eqref{pseudo} is not necessarily uniform over time as it is in the usual shadowing type results.  For instance, if $\varepsilon<\lambda$, the sequence $\delta(n)=e^{(\lambda-\varepsilon)|n|}$ is $e^{\lambda-\varepsilon}$-admissible. Consequently,  the error allowed in \eqref{pseudo} can even grow exponentially fast. In particular, Theorem \ref{t1} extends previous results even in the nonautonomous context and under a uniform exponential dichotomy assumption (see Section \ref{sec: settings} for applications to these settings). 
\end{remark}

As a consequence of the previous construction we get that the solution of \eqref{nde} is actually unique whenever we require that the deviation of the pseudotrajectory from the true one is small with respect to the adapted norm. Indeed, we have the following result. 
\begin{corollary}\label{cor: shad new norm}
Suppose we are in the hypothesis of Theorem \ref{t1} and let $(y_n)_{n\in \Z}$ be a sequence satisfying \eqref{pseudo}. Then, there exists a solution $(x_n)_{n\in \Z}$ of ~\eqref{nde} satisfying 
\begin{equation*}
\lVert x_n-y_n\rVert _{\sigma^n(\omega)} \le L\delta (n), \quad \text{for each $n\in \Z$.}
\end{equation*}
Moreover, this solution is unique.
\end{corollary}
\begin{proof}
Existence of such a solution follows readily from the proof of Theorem \ref{t1}. So, it remains to observe that it is unique.

Let $(x_n)_{n\in \Z}$ be a solution of ~\eqref{nde} associated to $(y_n)_{n\in \Z}$ by the ``existence part" and consider $\textbf w=(w_n)_{n\in \Z}$ given by $w_n=x_n-y_n$. We start observing that $\textbf w$ is a fixed point for the operator $T$ given by \eqref{eq: op T}. Indeed, 
\begin{displaymath}
\begin{split}
\left( S(\mathbf w) \right)_n &=  g_{n-1}(w_{n-1}) \\
&=f_{\sigma^{n-1}(\omega)}(w_{n-1}+y_{n-1})-f_{\sigma^{n-1}(\omega)}(y_{n-1})+F_{\sigma^{n-1}(\omega)}(y_{n-1})-y_n\\
&=f_{\sigma^{n-1}(\omega)}(x_{n-1})+A(\sigma^{n-1}(\omega))y_{n-1}-y_n\\
&=f_{\sigma^{n-1}(\omega)}(x_{n-1})+A(\sigma^{n-1}(\omega))x_{n-1}-A(\sigma^{n-1}(\omega))w_{n-1}-y_n\\
&=x_n-A(\sigma^{n-1}(\omega))w_{n-1}-y_n\\
&=w_n-A(\sigma^{n-1}(\omega))w_{n-1}.
\end{split}
\end{displaymath}
Consequently, using property \eqref{proj},
\begin{displaymath}
\begin{split}
& \left(\Gamma_\omega(S(\mathbf w))\right)_n  \\
&=\sum_{k=0}^\infty \cA(\sigma^{n-k}(\omega),k)\Pi(\sigma^{n-k}(\omega))(S(\mathbf w))_{n-k}  \\
&-\sum_{k=1}^\infty \cA(\sigma^{n+k}(\omega),-k)\Pi(\sigma^{n+k}(\omega))(S(\mathbf w))_{n+k}  \\
&=\sum_{k=0}^\infty \cA(\sigma^{n-k}(\omega),k)\Pi(\sigma^{n-k}(\omega))\left(w_{n-k}-A(\sigma^{n-k-1}(\omega))w_{n-k-1}\right)  \\
&-\sum_{k=1}^\infty \cA(\sigma^{n+k}(\omega),-k)\left(\Id-\Pi(\sigma^{n+k}(\omega))\right)\left(w_{n+k}-A(\sigma^{n+k-1}(\omega))w_{n+k-1}\right)   \\
&= \Pi(\sigma^{n}(\omega))w_{n}+\left(\Id-\Pi(\sigma^{n}(\omega))\right)w_{n}\\
&= w_n.
\end{split}
\end{displaymath}
Therefore, recalling that $T(\textbf w)=\Gamma_\omega(S(\textbf w))$, it follows that
\begin{displaymath}
T(\textbf w) = \mathbf w
\end{displaymath}
as claimed.

Moreover,
\begin{displaymath}
\begin{split}
\|\textbf w\|_{\delta,\infty}&=\sup_{n\in \Z} (\delta (n)^{-1}\lVert w_n \rVert_{\sigma^n (\omega)})\\
&=\sup_{n\in \Z} (\delta (n)^{-1}\lVert x_n-y_n \rVert_{\sigma^n (\omega)})\\
&\leq L.
\end{split}
\end{displaymath}
Consequently, since $T$ is a contraction from $ D=\{ \mathbf z\in Y_{\delta, \infty}: \lVert \mathbf z\rVert_{\delta, \infty} \le L\}$ to itself and, in particular, its fixed point in $D$ is unique, the result follows.
\end{proof}

\begin{corollary}[Expansivity] \label{cor: expansivity}
Suppose that the hypothesis of Theorem~\ref{t1} hold and let $(x_n)_{n\in \Z}$ and $(y_n)_{n\in \Z}$ be solutions of ~\eqref{nde} satisfying 
\begin{equation*}
\lVert x_n-y_n\rVert _{\sigma^n(\omega)} \le L\delta (n), \quad \text{for each $n\in \Z$.}
\end{equation*}
Then, $(x_n)_{n\in \Z}=(y_n)_{n\in \Z}$.
\end{corollary}
\begin{proof}
Obviously $(y_n)_{n\in \Z}$ satisfies \eqref{pseudo} and moreover it is shadowed by itself and by $(x_n)_{n\in \Z}$. Thus, uniqueness given by the previous corollary implies $(x_n)_{n\in \Z}=(y_n)_{n\in \Z}$ as claimed.
\end{proof}

\begin{corollary}\label{cor1}
Assume that $\cA$ admits a tempered exponential dichotomy and let $\Omega'\subset \Omega$, $\lambda >0$ and $K\colon \Omega \to (0, \infty)$ be as in the Definition~\ref{xxv}. Furthermore, suppose that $c\ge 0$ satisfies
\begin{equation}\label{contraction2}
2c\frac{1+e^{-\lambda}}{1-e^{-\lambda}}<1,
\end{equation}
and that~\eqref{non} holds.
There exists $L=L(\lambda, c)>0$ such that for any $t>0$,   $\omega \in \Omega'$ and a sequence  $(y_n)_{n\in \Z}\subset X$ such that
\[
\lVert y_n-F_{\sigma^{n-1}(\omega)}(y_{n-1})\rVert \le \frac{t}{2K(\sigma^n (\omega))} \quad \text{for $n\in \Z$,}
\]
then there exists a solution $(x_n)_{n\in \Z}$ of~\eqref{nde} with the property that
\[
\lVert x_n-y_n\rVert \le Lt, \quad \text{for each $n\in \Z$.}
\]
\end{corollary}

\begin{proof}
It only remains to apply Theorem~\ref{t1} in the particular case when $\delta \colon \Z \to (0, \infty)$ is a constant map $\delta(n)=t$, $n\in \Z$. Indeed, observe that $\delta$ is an $1$-admissible sequence and thus the desired conclusion follows from Theorem~\ref{t1} applied  for $\epsilon=\lambda$ (observe that in this case~\eqref{contraction} and~\eqref{contraction2} 
coincide).
\end{proof}

We stress that Theorem~\ref{t1} in particular applies to linear dynamics
\begin{equation}\label{ldyn}
x_{n+1}=A(\sigma^n (\omega))x_n, \quad n\in \Z.
\end{equation}

\begin{corollary}
Assume that $\cA$ admits a tempered exponential dichotomy and let $\Omega'\subset \Omega$, $\lambda >0$ and $K\colon \Omega \to (0, \infty)$ be as in the Definition~\ref{xxv}.
Furthermore, suppose that $0<\epsilon \le \lambda$. 
Then, there exists $L=L(\lambda, \epsilon)>0$ such that for any  $\omega \in \Omega'$ and a sequence  $(y_n)_{n\in \Z}\subset X$ such that
\begin{equation*}
\lVert y_n-A(\sigma^{n-1}(\omega))y_{n-1}\rVert \le \frac{\delta (n)}{2K(\sigma^n (\omega))} \quad \text{for $n\in \Z$,}
\end{equation*}
 there is a solution $(x_n)_{n\in \Z}$ of~\eqref{ldyn} with the property that~\eqref{shad} holds.
\end{corollary}

\begin{proof}
The desired conclusion follows directly from Theorem~\ref{t1} applied to the case when $f_\omega=0$. In this case $c=0$ and consequently~\eqref{contraction} is trivially satisfied. 
\end{proof}
We also have the following version of Corollary~\ref{cor1}.
\begin{corollary}\label{2:59}
Assume that $\cA$ admits a tempered exponential dichotomy and let $\Omega'\subset \Omega$, $\lambda >0$ and $K\colon \Omega \to (0, \infty)$ be as in the Definition~\ref{xxv}.
Then, there exists $L=L(\lambda)>0$ such that for 
any $t>0$,  $\omega \in \Omega'$ and a sequence  $(y_n)_{n\in \Z}\subset X$ such that
\[
\lVert y_n-A(\sigma^{n-1}(\omega))y_{n-1} \rVert \le \frac{t}{2K(\sigma^n (\omega))} \quad \text{for $n\in \Z$,}
\]
then there exists a solution $(x_n)_{n\in \Z}$ of~\eqref{ldyn} with the property that
\[
\lVert x_n-y_n\rVert \le Lt, \quad \text{for each $n\in \Z$.}
\]
\end{corollary}

\section{Conservation of Lyapunov exponents} \label{sec: conservation}
In this section we present an application of our results to the theory of Lyapunov exponents. More precisely, we prove that Lyapunov exponents associated with a random linear dynamics admitting a tempered exponential dichotomy remain unchanged under small nonlinear perturbations.

In order to simplify exposition, we will restrict ourselves to the case when $X=\mathbb{R}^d$ and $A:\Omega \to GL(d,\R)$. Moreover, we consider the linear cocycle $\cA \colon \Omega \times \Z \to GL(d,\R)$ defined by
\begin{displaymath}
\cA(\omega, n)=\begin{cases}
A(\sigma^{n-1}(\omega)) \cdots A(\sigma (\omega))A(\omega)  & \text{if $n\geq 1$;}\\
\Id & \text{if $n=0$;}\\
A(\sigma^{-|n|}(\omega))^{-1}\cdots A(\sigma^{-2}(\omega))^{-1}A(\sigma^{-1} (\omega))^{-1}  & \text{if $n<0$.}
\end{cases}
\end{displaymath}

Assuming $\log ^+\|A(\omega)^{\pm 1}\|\in L^1(\mathbb{P})$, it follows from the \emph{Oseledets theorem}~\cite{Osel} that there exist numbers $\lambda _1(\cA,\mathbb{P})>\ldots > \lambda _{k}(\cA,\mathbb{P})$, called the \emph{Lyapunov exponents}, and a decomposition $\mathbb{R}^d=E^1_{\omega}\oplus \ldots \oplus E^k_{\omega}$, called the \emph{Oseledets splitting}, into vector subspaces depending measurably on $\omega$ such that for $\mathbb{P}$-almost every $\omega \in \Omega$,
\begin{equation}\label{eq: Lyap exp}
A(\omega)E^i_{\omega}=E^i_{\sigma(\omega)} \; \textrm{and} \; \lambda _i(\cA,\mathbb{P}) =\lim _{n\to \pm  \infty} \dfrac{1}{n}\log \| \cA(\omega,n)x\| 
\end{equation}
for every non-zero $x\in E^i_{\omega}$ and $1\leq i \leq k$ (see for instance \cite{Via14}). By diminishing $\Omega '$ given in Definition \ref{xxv}, if necessary, we may assume that all these claims hold true for every $\omega\in \Omega'$.

From now on assume we are in the hypothesis of Theorem \ref{t1}. Moreover, assume the sequence $\delta:\Z\to (0,+\infty)$ grows sub-exponentially, that is,
\begin{equation}\label{eq: sub-exp}
\lim_{n\to \pm \infty} \frac{1}{n}\log \delta(n) =0
\end{equation}
and that the nonlinear perturbations $f_\omega :\R^d \to \R^d$, $\omega \in \Omega$, are small in the sense that
\begin{equation}\label{eq: bound f Lyap}
\|f_{\sigma^n(\omega)}(x)\|\le \frac{\delta(n)}{4K(\sigma^n(\omega))} 
\end{equation} 
for every $\omega \in \Omega$ and $x\in \R^d$.  Observe that if $r:=\inf_{n\in \Z}\delta(n)>0$, then condition \eqref{eq: bound f Lyap} holds whenever
\begin{equation*}
\|f_{\omega}(x)\|\le \frac{r}{4K(\omega)},
\end{equation*} 
for every $\omega \in \Omega$ and $x\in \R^d$. We refer to Section \ref{sec: settings} for some examples where the above condition holds. Decreasing the constant $c$ given in \eqref{non}, if necessary, we have that $F_\omega=A(\omega)+f_\omega$ is a homeomorphism (see \cite[Section 4.1]{BD19}) and thus we can consider the cocycle
\begin{displaymath}
\mathcal{F}(\omega, n)=\begin{cases}
F_{\sigma^{n-1}(\omega)} \circ \cdots \circ F_{\sigma (\omega)} \circ F_{\omega}  & \text{if $n\geq 1$;}\\
\Id & \text{if $n=0$;}\\
F_{\sigma^{-|n|}(\omega)}^{-1}\circ \cdots \circ F_{\sigma^{-2}(\omega)}^{-1}\circ F_{\sigma^{-1} (\omega)}^{-1}  & \text{if $n<0$.}
\end{cases}
\end{displaymath}

We define the \emph{forward and backward Lyapunov exponents of $\mathcal{F}$ at $\omega$ in the direction $x\in \R^d$}, respectively, by  
\begin{displaymath}
\lambda^+ (\mathcal{F},\omega,x) =\limsup_{n\to + \infty} \dfrac{1}{n}\log \| \mathcal{F}(\omega,n)x\| 
\end{displaymath}
and
\begin{displaymath}
\lambda^- (\mathcal{F},\omega,x) =\limsup_{n\to - \infty} \dfrac{1}{n}\log \| \mathcal{F}(\omega,n)x\| .
\end{displaymath}

\begin{theorem} \label{theo: conservation}
For every $i=1,2,\ldots,k$ there exist $x\in \R^d$, $\omega\in \Omega'$ and $*\in \{+,-\}$ so that  
$$\lambda_i(\cA,\mathbb{P})=\lambda^* (\mathcal{F},\omega,x).$$ 
Reciprocally, for every $\omega\in \Omega'$ there exists $p:=p(\omega)\in \R^d$ so that for every $x\in \R^d\setminus \{p\}$ there exists $i\in \{1,2,\ldots,k\}$ such that 
$$\lambda^+ (\mathcal{F},\omega,x)=\lambda_i(\cA,\mathbb{P}) \text{ or } \lambda^- (\mathcal{F},\omega,x)=\lambda_i(\cA,\mathbb{P}).$$

\end{theorem}

In other words, what we are proving is that all the Lyapunov exponents of $\mathcal{A}$ are also Lyapunov exponents of $\mathcal{F}$ and, reciprocally, at least one of the Lyapunov exponents $\lambda^+ (\mathcal{F},\omega,x)$ and $ \lambda^- (\mathcal{F},\omega,x)$, $\omega\in \Omega'$ and $x\in \mathbb{R}^d\setminus \{p\}$, of $\mathcal{F}$ is also a Lyapunov exponent of $\mathcal{A}$.  To the best of our knowledge, no such result appeared previously in the literature.

\begin{proof} Given $i\in \{1,2,\ldots , k\}$, let $\omega\in \Omega'$ and $x\in E^i_{\omega}$ be such that 
\begin{equation}\label{eq: aux 1 LE}
\lambda _i(\cA,\mathbb{P}) =\lim _{n\to \pm  \infty} \dfrac{1}{n}\log \| \cA(\omega,n)x\|.
\end{equation}
Since $\cA$ admits a tempered exponential dichotomy, we have that all the Lyapunov exponents of $\cA$ are nonzero. Assume for the moment that $\lambda_i(\cA,\mathbb{P})<0$. Considering $(x_n)_{n\in \Z}$ given by $x_n=\cA(\omega,n)x$, we have that $x_{n+1}=A(\sigma^n(\omega))x_n$ for evey $n \in \Z$. Moreover, 
\begin{displaymath}
\|x_{n+1}-F_{\sigma^n(\omega)}(x_n)\|=\|-f_{\sigma^n(\omega)}(x_n)\|\le \frac{\delta(n)}{4K(\sigma^n(\omega))}.
\end{displaymath}
In particular, $(x_n)_{n\in \Z}$ satisfies \eqref{pseudo} and thus, by Theorem \ref{t1}, there exists a sequence $(y_n)_{n\in \Z}$ satisfying $y_{n+1}=F_{\sigma^n(\omega)}(y_n)$, $n\in \Z$, so that 
\begin{equation}\label{eq: aux 2 LE}
\|x_n-y_n\|\le L\delta(n) \text{ for every } n\in \Z.
\end{equation}
We are going to observe now that $\lambda_i(\cA,\mathbb{P})=\lambda^- (\mathcal{F},\omega,y_0)$. It follows from \eqref{eq: aux 1 LE} that
$$-\lambda _i(\cA,\mathbb{P}) =\lim _{n\to +\infty} \dfrac{1}{n}\log \| \cA(\omega,-n)x\|>0$$
which can be rewritten as 
$$-\lambda _i(\cA,\mathbb{P}) =\lim _{n\to +\infty} \dfrac{1}{n}\log \| x_{-n}\|>0.$$
Thus, using \eqref{eq: sub-exp} and \eqref{eq: aux 2 LE} it follows that
$$\lim _{n\to +\infty} \dfrac{1}{n}\log \| y_{-n}\|=-\lambda_i(\cA,\mathbb{P})$$
and consequently,
$$\lambda^{-} (\mathcal{F},\omega,y_0)=\limsup_{n\to -\infty} \dfrac{1}{n}\log \| \mathcal{F}(\omega,n)y_0\|=\lambda_i(\cA,\mathbb{P}).$$
 Similarly,  if $\lambda_i(\cA,\mathbb{P})>0$, we  obtain that
$$\lambda^+ (\mathcal{F},\omega,y_0)=\limsup_{n\to +\infty} \dfrac{1}{n}\log \| \mathcal{F}(\omega,n)y_0\|=\lambda_i(\cA,\mathbb{P}).$$

We now prove the converse statement by proceeding similarly to what we did above. Let $x\in \R^d$ and $\omega\in \Omega'$ be given. Considering $(x_n)_{n\in \Z}$ given by $x_n=\mathcal{F}(\omega,n)x$, we have that $x_{n+1}=F_{\sigma^n(\omega)}(x_n)$, $n\in \Z$. Moreover,  
\begin{displaymath}
\|x_{n+1}-A(\sigma^n(\omega))x_n\|=\|f_{\sigma^n(\omega)}(x_n)\|\le \frac{\delta(n)}{4K(\sigma^n(\omega))}.
\end{displaymath}
In particular, $(x_n)_{n\in \Z}$ satisfies \eqref{pseudo} for $\delta'(n)=\frac{\delta(n)}{2}$ instead of $\delta$ and  in the case when $f_\omega\equiv 0$ for every $\omega\in \Omega$. By Corollary \ref{cor: shad new norm}, there exists a unique sequence $(y_n)_{n\in \Z}$ satisfying $y_{n+1}=A(\sigma^n(\omega))y_n$ for $n\in \Z$, and such that 
\begin{equation}\label{eq: aux 3 LE}
\|x_n-y_n\|_{\sigma^n(\omega)}\le \frac{L\delta(n)}{2} \text{ for every } n\in \Z.
\end{equation}
We are now in a position to construct the point $p$ from the statement of the theorem. 

\begin{lemma}\label{claim}
There exists a unique point $p\in \mathbb{R}^d$, for which the sequence $(y_n)_{n\in \Z}$ given by \eqref{eq: aux 3 LE} satisfies $y_n=0$ for every $n\in \Z$.  
\end{lemma}
\begin{proof}[Proof of the Lemma~\ref{claim}]
We start observing that, if such a point does exist, then it is unique. Indeed, given $z,x\in \mathbb{R}^d$ let us consider $(x_n)_{n\in \Z}$ and $(z_n)_{n\in \Z}$ given by $x_n=\mathcal{F}(\omega,n)x$ and $z_n=\mathcal{F}(\omega,n)z$, respectively. By the previous construction we know that each of these sequences is shadowed by an actual orbit of $\cA(\omega, \cdot)$. Suppose both $(x_n)_{n\in \Z}$ and $(z_n)_{n\in \Z}$ are shadowed by the null sequence $(0)_{n\in \Z}$. That is, $\|x_n-0\|_{\sigma^n (\omega)}\le \frac{L\delta(n)}{2}$ and $\|z_n-0\|_{\sigma^n (\omega)}\le \frac{L\delta(n)}{2}$ for every $n\in \Z$. Then, $\|x_n-z_n\|_{\sigma^n (\omega)}\le L\delta(n)$ for every $n\in \Z$. It follows from Corollary~\ref{cor: expansivity} that $x_n=z_n$ for $n\in \Z$. In particular, we have that $x=x_0=z_0=z$ as claimed. 
 Existence of this point can be proved by observing that the null sequence $(0)_{n\in \Z}$ satisfies \eqref{pseudo} and applying Corollary \ref{cor: shad new norm}. By doing so, we obtain a sequence $(x_n)_{n\in \Z}\subset X$ such that $x_n=\mathcal F(\omega, n)x_0$ for $n\in \Z$ and
$\| x_n\|_{\sigma^n(\omega)}\le \frac{L\delta(n)}{2}$ for $n\in \Z$. Thus, $p=x_0$ satisfies the desired conclusion.
\end{proof}

Returning  back to the proof of the theorem and assuming that $x\neq p$, we have  (since $(y_n)_{n\in \Z}$ is not the null sequence)  that 
\begin{displaymath}
\lim_{n\to +\infty}\frac{1}{n}\log \|y_n\|=\lim_{n\to +\infty}\frac{1}{n}\log \|\cA (\omega,n)y_0\|=\lambda_i(\cA,\mathbb{P})
\end{displaymath}
for some $i\in \{1,2,\ldots,k\}$. Now, since $\cA$ admits a tempered exponential dichotomy it follows that $\lambda_i(\cA,\mathbb{P})\neq 0$. If $\lambda_i(\cA,\mathbb{P})>0$ then using \eqref{ln}, \eqref{eq: sub-exp} and \eqref{eq: aux 3 LE}, we conclude that 
\begin{displaymath}
\lambda_i(\cA,\mathbb{P})=\lim_{n\to +\infty}\frac{1}{n}\log \|x_n\|=\lim_{n\to +\infty}\frac{1}{n}\log \|\mathcal{F} (\omega,n)x\|=\lambda^+(\mathcal{F},\omega,x).
\end{displaymath}
 On the other hand, if $\lambda_i(\cA,\mathbb{P})<0$ then we similarly conclude that 
\begin{displaymath}
\lambda_i(\cA,\mathbb{P})=\lambda^-(\mathcal{F},\omega,x).
\end{displaymath}
The proof of the theorem is completed.
\end{proof}

\begin{remark}
We observe that the conclusion of Theorem \ref{theo: conservation} is sharp in the sense that we may have $\lambda^+ (\mathcal{F},\omega,x)\neq \lambda^- (\mathcal{F},\omega,x)$ and, consequently, we don't necessarily have $\lambda^+ (\mathcal{F},\omega,x)=\lambda_i(\cA,\mathbb{P})= \lambda^- (\mathcal{F},\omega,x)$ in the reciprocal part. Indeed, let $(\Omega, \mathcal F, \mathbb P)$ be a probability space and $\sigma \colon \Omega \to \Omega$ an ergodic $\mathbb P$-preserving invertible transformation. Assume moreover that $\mathbb{P}$ is non-atomic. Consider $A:\Omega\to GL(1,\mathbb{R})$ given by $A(\omega)=A$ for every $\omega\in \Omega$ where $A\colon \mathbb{R} \to \mathbb{R}$ is given by $Ax=\frac{1}{2} x$ for every $x\in \R$. Let $\cA:\Omega \times \Z\to GL(1,\mathbb{R})$ be the cocycle generated by $A$ as in the beginning of this section. Observe that this cocycle admits a tempered exponential dichotomy with $\Omega'=\Omega$, $\Pi(\omega)=\Id$ and $K(\omega)=1$ for every $\omega\in \Omega$ and $\lambda=\log 2$. Fix a non-periodic point $\omega\in \Omega$ and $\tau>0$. Now, for each $\omega'\in \Omega$ let us consider $f_{\omega'}:\mathbb{R}\to \mathbb{R}$ given by $f_{\omega'}(x)=\tau$ for every $x\in \R$ if $\omega'\in \{\sigma^n(\omega)\}_{n\geq 0}$ and $f_{\omega'}\equiv0$ otherwise. It is now easy to see that, whenever $\tau$ is sufficiently small, the hypothesis of Theorems \ref{t1} and \ref{theo: conservation} are satisfied and, moreover, given $x\in \R\setminus\{0\}$, $\lambda^- (\mathcal{F},\omega,x)=-\log 2$ while $\lambda^+ (\mathcal{F},\omega,x)=0$.
\end{remark}

\section{Applicable Settings} \label{sec: settings}

The hypothesis on the perturbations allowed in our main results are written in terms of constants coming from the hyperbolicity.
For instance, condition \eqref{non} requires the perturbations $f_\omega$ to be Lipschitz. In addition, the Lipschitz  constants of maps  $f_\omega$ depend on the strength of the hyperbolicity. In this section we will present some settings where these hypothesis can be easily fulfilled and thus, our results can be applied.

\subsection{Uniformly hyperbolic systems}
We start by applying our main results to the case when $\cA$ is uniformly hyperbolic. 
\begin{definition}\label{2x44}
We say that $\cA$ admits a  \emph{uniform exponential dichotomy} if there exist $K, \lambda >0$ and a family of projections $\Pi(\omega)$, $\omega \in \Omega$ such that for every $\omega \in \Omega$:
\begin{enumerate}
\item $\Pi(\sigma^n (\omega))\cA(\omega, n)=\cA(\omega, n)\Pi(\omega)$ for $n\in \mathbb N$;
\item for $n\in \mathbb N$,
\[
\cA(\omega, n)\rvert_{\Ker \Pi(\omega)} \colon \Ker \Pi(\omega) \to \Ker \Pi(\sigma^n (\omega))
\]
is invertible;
\item 
\begin{equation}\label{td111}
\lVert \cA(\omega, n)\Pi(\omega)\rVert \le Ke^{-\lambda n}, \quad n\ge 0
\end{equation}
and
\begin{equation}\label{td22}
\lVert \cA(\omega, -n)(\Id-\Pi(\omega))\rVert \le Ke^{-\lambda n}, \quad n\ge 0,
\end{equation}
where 
\[
\cA(\omega, -n):=\bigg{(}\cA(\sigma^{-n}(\omega), n)\rvert_{\Ker \Pi (\sigma^{-n} (\omega))} \bigg{)}^{-1}.
\]

\end{enumerate}
\end{definition}
Then, we have the following consequence of Theorem~\ref{t1}.

\begin{corollary}\label{cor: unif}
Assume that $\cA$ admits a uniform exponential dichotomy and let  $\lambda, K >0$  be as in the Definition~\ref{2x44}.  Furthermore, suppose that $\epsilon>0$, $c\ge 0$ are such that $\epsilon \le \lambda$ and that~\eqref{contraction} holds. 
Finally, we assume that
\begin{equation*}
\lVert f_\omega(x)-f_\omega(y)\rVert \le \frac c K \lVert x-y\rVert, \quad \text{for $\omega \in \Omega$ and $x, y\in X$.}
\end{equation*}
Then, there exists $L=L(\lambda, \epsilon, c)>0$ such that for every  $e^{\lambda-\epsilon}$-admissible sequence 
 $\delta \colon \Z \to (0, \infty)$,   $\omega \in \Omega$ and a sequence  $(y_n)_{n\in \Z}\subset X$ such that
\begin{equation}\label{pseudo3}
\lVert y_n-F_{\sigma^{n-1}(\omega)}(y_{n-1})\rVert \le \delta (n) \quad \text{for $n\in \Z$,}
\end{equation}
 there is a solution $(x_n)_{n\in \Z}$ of~\eqref{nde} such that~\eqref{shad} holds. 
\end{corollary}
\begin{proof}
One only needs to apply Theorem~\ref{t1} in the case when $\Omega'=\Omega$, $K(\omega)=K$ and replace $\delta$ with $n\mapsto 2K\delta(n)$.
\end{proof}
 As observed in Remark \ref{remark 2}, due to the flexibility of \eqref{pseudo3}, Corollary \ref{cor: unif} extends previous results (even under a uniform exponential dichotomy assumption).

As in the general case, our results in particular apply to linear dynamics~\eqref{ldyn}. We shall formulate only a version of Corollary~\ref{2:59} in this context.
\begin{corollary}\label{ff}
Assume that $\cA$ admits a uniform exponential dichotomy and let $\lambda, K >0$  be as in the Definition~\ref{2x44}.
Then, there exists $L=L(\lambda)>0$ such that for 
any $t>0$,  $\omega \in \Omega$ and a sequence  $(y_n)_{n\in \Z}\subset X$ such that
\[
\lVert y_n-A(\sigma^{n-1}(\omega))y_{n-1} \rVert \le t \quad \text{for $n\in \Z$,}
\]
then there exists a solution $(x_n)_{n\in \Z}$ of~\eqref{ldyn} with the property that
\[
\lVert x_n-y_n\rVert \le Lt, \quad \text{for each $n\in \Z$.}
\]
\end{corollary}

\begin{remark}
The result obtained in Corollary~\ref{ff} can be described as a Hyers-Ulam stability result for the  random linear dynamics given by~\eqref{ldyn} under the assumption that it admits a uniform exponential dichotomy. 
\end{remark}

Let us now obtain a partial converse to Corollary~\ref{ff}.

\begin{proposition}
Assume that $\cA$ is an invertible cocycle, i.e. that $A(\omega)$ is an invertible operator for each $\omega \in \Omega$. Furthermore, suppose that  there exists $L>0$ such that for each sequence  $(y_n)_{n\in \Z}\subset X$ such that
\begin{equation}\label{340}
\lVert y_n-A(\sigma^{n-1}(\omega))y_{n-1} \rVert \le 1 \quad \text{for $n\in \Z$,}
\end{equation}
then there exists a solution $(x_n)_{n\in \Z}$ of~\eqref{ldyn} with the property that
\begin{equation}\label{opp}
\lVert x_n-y_n\rVert \le L, \quad \text{for each $n\in \Z$.}
\end{equation}
Finally, assume that~\eqref{ldyn} has no bounded solutions. 
Then, $\cA$ admits a uniform exponential dichotomy. 
\end{proposition}

\begin{proof}
Let us fix $\omega \in \Omega$ and take an arbitrary $\mathbf z=(z_n)_{n\in \Z}\subset X$ such that $\lVert \mathbf z\rVert_\infty:=\sup_{n\in \Z} \lVert z_n\rVert<\infty$. Take a sequence $(y_n)_{n\in \Z}\subset X$ such that
\[
y_n-A(\sigma^{n-1}(\omega))y_{n-1}=\frac{1}{\lVert \mathbf z\rVert_\infty}z_n, \quad \text{for $n\in \Z$.}
\]
Observe that $(y_n)_{n\in \Z}$ satisfies~\eqref{340}.  Hence, there exists a solution $(x_n)_{n\in \Z}$ of~\eqref{ldyn} satisfying~\eqref{opp}. Consequently, the sequence $\mathbf w=(w_n)_{n\in \Z}\subset X$ given by
\[
w_n=\lVert \mathbf z\rVert_\infty (y_n-x_n) \quad n\in \Z,
\]
is a solution of~\eqref{ldyn}. In addition,
\[
\lVert \mathbf w\rVert_\infty =\sup_{n\in \Z} \lVert w_n\rVert \le L\lVert \mathbf z\rVert_\infty. 
\]
Hence, by applying results from~\cite{CL},  we conclude that $\cA$ admits a uniform exponential dichotomy. 
\end{proof}

\subsubsection{Theorem \ref{theo: conservation} for uniformly hyperbolic systems} Suppose we are in the setting of Section \ref{sec: conservation} and that  the assumptions of Corollary \ref{cor: unif} are satisfied. Let $\delta:\Z\to (0,+\infty)$ be a sequence satisfying~\eqref{eq: sub-exp} and such that $r:=\inf_{n\in \Z} \delta(n)>0$. For instance, we can take $\delta$ to be a constant sequence or we can take $\delta$ given by $\delta(n)=|n|+1$, $n\in \Z$. Let $f_\omega:\mathbb{R}^d\to \mathbb{R}^d$, $\omega\in \Omega$, be such that
\begin{displaymath}
\|f_\omega(x)\|\leq \frac{r}{4K} \text{ for every } x\in \mathbb{R}^d.
\end{displaymath}
 Observe that the above condition implies that \eqref{eq: bound f Lyap} holds and consequently  Theorem \ref{theo: conservation} can be applied to this setting.

\subsection{Nonuniformly hyperbolic systems}\label{NHS}
Suppose that a cocycle $\cA: \Omega \times \mathbb{N}_0 \to X$ admits a tempered exponential dichotomy and let $K\colon \Omega \to (0, \infty)$ be the tempered random variable from Definition~\ref{xxv}. It follows from \cite[Proposition 4.3.3 ii)]{A} that for every $\rho >0$ there exists a random variable $D=D_\rho: \Omega\to (0,+\infty)$ such that
\begin{equation}\label{eq: conseq temp}
 K(\omega)\leq D(\omega) \text{ and } D(\sigma^n(\omega))\leq D(\omega )e^{\rho |n|}, 
\end{equation}
for $\mathbb P$-a.e. $\omega \in \Omega$ and $n\in \Z$.

Take now an arbitrary $\rho>0$ and consider the corresponding random variable $D=D_\rho:\Omega\to (0,+\infty)$ satisfying~\eqref{eq: conseq temp}. Given $T>0$, let us consider
\begin{displaymath}
\Omega_T'=\{\omega\in \Omega'; D(\omega)\le T\}.
\end{displaymath}
Noting that $\lim_{T\to \infty} \mathbb P(\Omega_T')=1$, we can fix $T$ sufficiently large so that $\mathbb P(\Omega_T')>0$. 
For $n\in \mathbb N$, set  \[\Omega_T^n:=\sigma^{-n}(\Omega_T')\setminus \cup_{k=0}^{n-1}\sigma^{-k}(\Omega_T').\] In addition, let $\Omega_T^0:=\Omega_T'$. Then, the ergodicity of the base system $(\Omega, \mathcal F, \mathbb P, \sigma)$ implies that  $\mathbb{P}\left(\cup_{n=0}^\infty \Omega_T^n\right)=1$. Moreover, observe that 
$\Omega_T^n \cap \Omega_T^m=\emptyset$ for $n\neq m$.

For $n\ge 0$ and $\omega\in \Omega_T^n$,  let $f_\omega:X\to X$ be such that 
\begin{equation}\label{eq: lip f ex}
\|f_\omega(x)-f_\omega(y)\|\leq \frac{c}{T}e^{-\rho|n-1|}\|x-y\| \text{ for every } x,y \in X,
\end{equation}
where $c$ is as in the statement of Theorem~\ref{t1}.
It is easy to see that \eqref{eq: lip f ex} combined with \eqref{eq: conseq temp} implies that \eqref{non} is satisfied.  Indeed, if $n\ge 1$ we have that $\omega=\sigma^{-n}(\omega')$ for some $\omega'\in \Omega_T'$. Therefore,
\[
K(\sigma(\omega))\le D(\sigma(\omega))=D(\sigma^{-(n-1)}(\omega'))\le e^{\rho |n-1|}D(\omega')\le Te^{\rho |n-1|},
\]
which gives that 
\[
\frac{c}{T}e^{-\rho |n-1|}\le \frac{c}{K(\sigma (\omega))}.
\]
Consequently, \eqref{eq: lip f ex} implies~\eqref{non}. One can argue similarly in the case when $n=0$.
We conclude that in the present setting Theorem~\ref{t1} is applicable. In addition, observe that if $\omega \in \Omega_T^m$ for some  $m\ge 0$, then each sequence $(y_n)_{n\in \Z}\subset X$ such that
\[
\lVert y_n-F_{\sigma^{n-1}(\omega)}(y_{n-1})\rVert \le \frac{\delta(n)}{2T}e^{-\rho|n-m|} \quad \text{for $n\in \Z$,}
\]
satisfies~\eqref{pseudo} (and consequently~\eqref{shad} holds).

\subsubsection{Theorem \ref{theo: conservation} for nonuniformly hyperbolic systems} Suppose we are in the setting of Section \ref{sec: conservation} and Subsection~\ref{NHS}. Furthermore, 
 let $\delta:\Z\to (0,+\infty)$ be a sequence satisfying~\eqref{eq: sub-exp} and such that  $r:=\inf_{n\in \Z}\delta(n)>0$. For each $\omega \in \Omega^n_T$, $n\ge 0$, let $f_\omega:\mathbb{R}^d\to \mathbb{R}^d$ be such that
\begin{displaymath}
\|f_\omega(x)\|\leq \frac{r}{4T}e^{-\rho n} \text{ for every } x\in \mathbb{R}^d.
\end{displaymath}
It is easy to see that the hypothesis of Theorem \ref{theo: conservation} are satisfied and thus we may apply it in the present setting.


\medskip{\bf Acknowledgements.} We would like to thank the anonymous referee for constructive comments that helped us to improve our paper. L.B. was partially supported by a CNPq-Brazil PQ fellowship under Grant No. 306484/2018-8. D. D. was supported in part by Croatian Science Foundation under the project
IP-2019-04-1239 and by the University of Rijeka under the projects uniri-prirod-18-9
and uniri-prprirod-19-16.

\end{document}